\documentclass[a4paper,12pt]{amsart}

\pdfoutput=1

\usepackage{amsmath,amssymb,amsthm,url}
\usepackage[utf8]{inputenc}
\usepackage[T1]{fontenc}
\usepackage{libertine}
\usepackage[libertine,cmintegrals,cmbraces]{newtxmath}
\usepackage{bbold}
\usepackage[shortlabels]{enumitem}
\usepackage{mathtools}
\usepackage{adjustbox}
\usepackage{microtype}
\usepackage{xcolor}
\usepackage{tikz}
\usepackage{tikz-cd}
\usepackage{nicefrac}
\usepackage{dsfont}
\usepackage{todonotes}
\usepackage{a4wide}
\usepackage{anyfontsize}

\AtBeginDocument{%
   \def\MR#1{}
} 
\usepackage{euscript}
\usepackage[all]{xy}

\usepackage[pagebackref]{hyperref}
  
\hypersetup{%
  bookmarksnumbered=true,
  colorlinks=true,%
  linkcolor=blue,%
  citecolor=blue,%
  filecolor=blue,%
  menucolor=blue,%
  urlcolor=blue,%
  pdfnewwindow=true,%
  pdfstartview=FitBH}
\usepackage[capitalise]{cleveref}

\usepackage{xcolor}
\definecolor{seagreen}{RGB}{46,139,87}
\definecolor{maroon}{RGB}{128,0,0}
\definecolor{darkviolet}{RGB}{148,0,211}
\definecolor{twelve}{RGB}{100,100,170}
\definecolor{thirteen}{RGB}{100,150,50}
\definecolor{fourteen}{RGB}{200,0,0}
\definecolor{fifteen}{RGB}{0,200,0}
\definecolor{sixteen}{RGB}{0,0,200}
\definecolor{seventeen}{RGB}{200,0,200}
\definecolor{eighteen}{RGB}{0,200,200}

\newcommand{\bb}[1]{\mathbb{#1}}

\newcommand{\es}[1]{\EuScript{#1}}
\renewcommand{\sf}[1]{{\mathsf{#1}}}

\newcommand{\fib}{\mathsf{fib}}

\newcommand{\s}{{\sf{Sp}}}

\newcommand{\homog}[1]{\mathsf{Homog}^{#1}}
\newcommand{\exc}[1]{\mathsf{Exc}_{\ast}^{ #1}}

\newcommand{\Aut}{\mathsf{Aut}}
\newcommand{\Fun}{\sf{Fun}}
\DeclareMathOperator{\Hom}{\mathsf{Hom}}

\DeclareMathOperator{\Map}{\mathsf{Map}}

  \newcommand{\adjunction}[4]{
\xymatrix{
#1:#2 \ar@<.5ex>[r] &
\ar@<.5ex>[l] #3:#4
}}

\newtheorem{thm}{Theorem}[section]
\newtheorem{prop}[thm]{Proposition}

\newtheorem{lem}[thm]{Lemma}
\newtheorem{cor}[thm]{Corollary}

\newtheorem{xxthm}{Theorem}

\newtheorem{xxcor}[xxthm]{Corollary}

\theoremstyle{definition}
\newtheorem{definition}[thm]{Definition}
\newtheorem{ex}[thm]{Example}

\newtheorem{rem}[thm]{Remark}

\begin{document}
\title{An algebraic model for rational excisive functors}

\author{David Barnes}
\address[Barnes]{Queen's University Belfast}
\email{d.barnes@qub.ac.uk}

\author{Magdalena K\k{e}dziorek}
\address[K\k{e}dziorek]{Radboud University Nijmegen}
\email{m.kedziorek@math.ru.nl}

\author{Niall Taggart}
\address[Taggart]{Queen's University Belfast}
\email{n.taggart@qub.ac.uk}

\begin{abstract}
We provide a new proof of the rational splitting of excisive endofunctors of spectra as a product of their homogeneous layers independent of rational Tate vanishing. We utilise the analogy between endofunctors of spectra and equivariant stable homotopy theory and as a consequence, we obtain an algebraic model for rational excisive functors.
\end{abstract}

\maketitle

\setcounter{tocdepth}{1}
{\hypersetup{linkcolor=black} \tableofcontents}

\section{Introduction}

Goodwillie calculus~\cite{GoodcalcI, GoodcalcII, GoodCalcIII} is a tool for studying functors between categories which carry a notion of homotopy. It approximates a homotopy-preserving functor $F$ by a tower 
\[\begin{tikzcd}
	&& F \\
	\cdots & {P_dF} & \cdots & {P_1F} & {P_0F}
	\arrow[from=2-4, to=2-5]
	\arrow[from=2-2, to=2-3]
	\arrow[from=2-3, to=2-4]
	\arrow[bend right=30, from=1-3, to=2-2]
	\arrow[from=2-1, to=2-2]
	\arrow[bend left=30, from=1-3, to=2-4]
	\arrow[bend left=20, from=1-3, to=2-5]
\end{tikzcd}\]
of functors, acting as a categorification of the Taylor series of a function from differential calculus. Through this analogy, polynomial functions of degree $d$ are replaced by $d$-excisive functors and the functor $P_dF$ appearing in the Goodwillie tower of $F$ is the best approximation to $F$ by a $d$-excisive functor. The difference between consecutive excisive approximations 
\[
D_d F = \fib(P_dF \longrightarrow P_{d-1}F),
\]
is called a homogeneous functor of degree $d$, analogous to a monomial of degree $d$. 

The basic aim of Goodwillie calculus is to understand the input functor $F$
by analysing its excisive and homogenous approximations. Hence, one
studies the structure of the categories of excisive and homogeneous functors
in general. 
In the case of $d$-homogeneous functors from spectra to spectra, Goodwillie's classification of homogeneous functors says that there is an equivalence of $(\infty-)$categories between the category of $d$-homogeneous functors and spectra with an action of $\Sigma_d$ (the group of permutations of $d$ elements)
\[
\homog{d}(\s^\omega,\s) \simeq \s^{B\Sigma_d},
\]
realised by sending a functor $F$ to its $d$-th derivative $\partial_d F \in \s^{B\Sigma_d}$. 
This is analogous to a monomial of degree $d$ being determined by its $d$-th derivative. 

The classification of categories of excisive functors is more complicated. 
For functors between the categories of (pointed) spaces and spectra, Arone and Ching~\cite{ACClassification} provided  models in terms of coalgebras over a comonad.
Glasman~\cite{Glasman} gives a simpler classification of the category $\exc{d}(\s^\omega,\s)$ of $d$-excisive endofunctors of spectra in terms of (spectral) Mackey functors on the category $\sf{Epi}_{\leq d}$ of epimorphisms of finite sets of cardinality at most $d$. 
This is analogous to the description of equivariant spectra for a finite group $G$ as (spectral) Mackey functors for the orbit category $\sf{Orb}_{G}$ of Guillou and May \cite{GM24} and Barwick \cite{barwick17}.  

Glasman's result instigated an analogy between equivariant homotopy theory and Goodwillie calculus which was utilised by Arone, Barthel, Heard and Sanders~\cite{AroneBarthelHeardSanders} to provide a tensor triangular classification of excisive functors through a computation of the Balmer spectrum of the category of (compact) $d$-excisive functors. 
A key ingredient in this work is the Goodwillie-Burnside ring $A(d)$: the abstract Burnside ring in the sense of Yoshida~\cite{YoshidaBurnside} associated to the category $\sf{Epi}_{\leq d}$. This is a generalisation of the Burnside ring from equivariant algebra and representation theory. Indeed, inputting the orbit category of a finite group $G$ into Yoshida's machine produces the familiar Burnside ring. 

In this paper, we analyse the idempotents of the rationalised Goodwillie-Burnside ring to give a new proof of the splitting of the category of rational $d$-excisive functors in terms of rational homogeneous functors.
In this context, a functor is rational if it takes values in $\bb{Q}$-local spectra, those spectra whose
homotopy groups are rational vector spaces. 

\begin{xxthm}\label{thm: main thm}
There is an equivalence
\[
\exc{d}(\s^\omega,\s)_\bb{Q} \simeq \prod_{1 \leq i \leq d} \homog{i}(\s^\omega,\s)_\bb{Q},
\]
of $\infty$-categories between the $\infty$-category of rational $d$-excisive functors and the product of the $\infty$-categories of rational homogeneous functors. 
\end{xxthm}

Results of this form are already present in the literature. For instance, Kuhn~\cite{Kuhn} and McCarthy~\cite{McCarthy} used the Kuhn-McCarthy pullback square together with the rational Tate vanishing for finite groups to prove that every $d$-excisive functor splits as a product of its homogeneous layers after rationalization. Our proof of~\cref{thm: main thm} (an outline is given at the end of the introduction) avoids both the use of the Kuhn-McCarthy pullback square and rational Tate vanishing. 

We can provide an even simpler description of $\exc{d}(\s^\omega,\s)_\bb{Q}$
using Goodwillie's classification of homogeneous functors (as described above)
and results from rational equivariant stable homotopy theory. 
Greenlees~\cite{greenleesreport} conjectured that the category of rational $G$-spectra (for $G$ any compact Lie group) has an algebraic model: an algebraic category $\es{A}(G)$ and an equivalence $\s^G_\bb{Q} \simeq \es{A}(G)$. This program has had many successes, for example \cite{BGKS, GreenleesShipleyTorus, Wimmer, BGK_elliptic, WIT3.1, WIT3.2}, see \cite{BarnesKedziorek} for a survey and introduction to the area. A key method in many of these results is using the idempotents of the rationalised Burnside ring $A_\bb{Q}(G)$. Applying the fact that the category of rational chain complexes with an action of $\Sigma_d$, $\sf{Ch}(\bb{Q}[\Sigma_d])$, is an algebraic model for $\s^{B\Sigma_d}$ gives the following. 

\begin{xxcor}\label{cor: main cor}
There is an equivalence
\[
\exc{d}(\s^\omega,\s)_\bb{Q} \simeq \prod_{1 \leq i \leq d} \sf{Ch}(\bb{Q}[\Sigma_i]),
\]   
of $\infty$-categories.
\end{xxcor}

We also obtain the following seemingly new result, as \cref{homogeneous modules}, describing rational 
homogenous functors as a category of modules. We do not expect this result to hold integrally. 

\begin{xxcor}\label{main module corollary}
The $d$-homogeneous functor $D_d(\Sigma^\infty\Omega^\infty)$ admits the structure of a commutative algebra object in the category of rational endofunctors of spectra, and there is an equivalence  
\[
\homog{d}(\s^\omega,\s)_\bb{Q} \simeq \sf{Mod}_{\Fun_\ast(\s^\omega,\s)_\bb{Q}}(D_d(\Sigma^\infty\Omega^\infty)),
\]
of $\infty$-categories.
\end{xxcor}

\subsection*{Outline of the proof}
Taking motivation from equivariant homotopy theory, the first step is to understand the rational Goodwillie-Burnside ring. Following Yoshida, we show (\cref{cor: splitting of rational ABR}) that the rational Goodwillie-Burnside ring splits
\[
A_\bb{Q}(d) \cong \prod_{1 \leq i \leq d}\bb{Q},
\]
by constructing an injective ring map $A_\bb{Q}(d) \to \prod_{1 \leq i \leq d}\bb{Q}$ using an epi-mono factorization of the category $\sf{Epi}_{\leq d}$. 

We can represent this ring map as a matrix. Computing the idempotents of the rational Goodwillie-Burnside ring then amounts to inverting this matrix (\cref{lem: idempotent}). In particular, for each $1 \leq i \leq d$ we obtain an idempotent $e_i$ of the Goodwillie-Burnside ring $A_\bb{Q}(d)$, corresponding to the projection on the $i$-th factor in the product $\prod_{1 \leq i \leq d}\bb{Q}$. 

By work of Arone, Barthel, Heard and Sanders~\cite{AroneBarthelHeardSanders}, the Goodwillie-Burnside ring is the endomorphism ring of the unit of the homotopy category of $d$-excisive functors (\cref{lem: endomorphism vs yoshida}), and hence these idempotents lift to the category of rational excisive functors, and induce a splitting of the category of rational excisive functors into an $e_d$-local piece and its orthogonal complement a $(1-e_d)$-local piece.

The final step in the proof is to identify the orthogonal pieces in the splitting. To do this, we compute the kernel of the natural map $A_\bb{Q}(d) \to A_\bb{Q}(d-1)$ coming from the inclusion of a subcategory $\sf{Epi}_{\leq d-1}$ into $\sf{Epi}_{\leq d}$, showing that it is generated by the top idempotent $e_d$ (\cref{ex:inclusion of subcategory epi}). This map on rational Goodwillie-Burnside rings corresponds to the universal map $P_dF \to P_{d-1}F$ (\cref{lem: Goodwillie burnside as endo ring}). It follows that for $F: \s \to \s$ a rational $d$-excisive functor, peeling off the top idempotent splits the Goodwillie tower of $F$ into two pieces,

\tikz[
overlay]{
    \filldraw[fill=yellow!50,draw=yellow!50] (6.7,-3.25) rectangle (12.5,0.25);
}
\tikz[
overlay]{
    \filldraw[fill=orange!50,draw=orange!50] (2.8,-3.25) rectangle (5.8,-1.25);

    \node(P0) at (4.3,-3) {$e_d$-local};
    \node(P0) at (9.5,-3) {$(1-e_d)$-local};

}
\[\begin{tikzcd}
	F && {P_{d-1}F} & \cdots & {P_1F} \\
	  {D_dF} && {D_{d-1}F} && {D_1F} \\
	\arrow[from=1-1, to=1-3]
	\arrow[from=1-3, to=1-4]
	\arrow[from=1-4, to=1-5]
	\arrow[from=2-1, to=1-1]
	\arrow[from=2-3, to=1-3]
	\arrow["\simeq"', from=2-5, to=1-5]
\end{tikzcd}\]

where the orange box is the $e_d$-local component and the yellow box is the $(1-e_d)$-local component. 
\cref{thm: splitting of exc} uses this to split the category of $d$-excisive functors
into $d$-homogeneous functors and $(d-1)$-excisive functors. The result follows by induction.

\subsection*{Contribution of this paper}

Our proof of the rational splitting of the category of excisive functors is very different to those currently in the literature. 
It relies mainly on the formal properties of the category as developed in \cite{AroneBarthelHeardSanders}, primarily 
the symmetric monoidal structure of the category and the structure of the rational Goodwillie-Burnside ring. 
The authors believe there should be further results along these lines, where one has some
generalisation of (topological) Mackey functors and can prove a splitting result
via properties of idempotents in the rational setting. 
Indeed, recent work of the first two authors with Hill on incomplete Mackey functors \cite{BHK24},
uses a similar splitting approach with rational idempotents and again connects to rational Tate vanishing.
The authors believe this approach will be valuable in a setting where the analogue of
rational Tate vanishing is not accessible. 

We believe that \cref{main module corollary} is new and of some independent interest, providing
a description of rational homogeneous functors in terms of a category of modules. 

The authors also want to add to the efforts of Arone, Barthel, Heard and Sanders~\cite{AroneBarthelHeardSanders}
in advertising the results and methods of Yoshida~\cite{YoshidaBurnside}. In this paper we concentrate on the rational case, and we explicitly describe and draw the categories $\sf{Epi}_{\leq 2}$, $\sf{Epi}_{\leq 3}$ and $\sf{Epi}_{\leq 4}$, calculate a complete set of idempotents for associated rational Goodwillie-Burnside rings in these cases and offer a comparison to the case of $\sf{Orb}_{G}$ for $G$ the cyclic group of order 6.

\subsection*{Acknowledgements}
We are grateful to Luca Pol for helpful conversations relating to this work and the anonymous referee for their useful comments. The authors would like to thank the Fondation des Treilles for its support and hospitality during the workshop ``Modelling symmetries using algebra'', where work on this paper was undertaken. The second and third author were supported by the Nederlandse Organisatie voor Wetenschappelijk Onderzoek (Dutch Research Council) Vidi grant no VI.Vidi.203.004. In the final stages of preparing this article for publication, the third author was supported by the Engineering and Physical Sciences Research Council under grant number EP/Z534705/1.

\section{Rational abstract Burnside rings}\label{sec: ABR}
In this section, we survey Yoshida's theory of abstract Burnside rings over the rational numbers~\cite{YoshidaBurnside}. It describes a process of assigning a ring to a finite skeletal category, which in the case of the orbit category of a finite group $G$ recovers (\cref{ex: C6}) the rational Burnside ring $A(G)$ from rational equivariant algebra~\cite[Example 3.3.(a)]{YoshidaBurnside}. We work rationally since our applications are rational, but the theory holds more generally, see e.g.,~\cite{YoshidaBurnside, YOT}.

Let $\es{C}$ denote a finite skeletal category, i.e., a category with finitely many objects, finitely many morphisms and such that if $c \cong d$, then $c=d$. Write $\bb{Q}^\es{C}$ for the product $\prod_{c \in \es{C}}\bb{Q}$ and denote by $A_\bb{Q}(\es{C})$ the free $\bb{Q}$-module generated by the set of object of $\es{C}$. Define the \emph{Burnside homomorphism} to be the $\bb{Q}$-linear map,
\[
\phi: A_\bb{Q}(\es{C}) \longrightarrow \bb{Q}^\es{C}, \ \ d \longmapsto (|\es{C}(c, d)|)_{c \in \es{C}},
\]
where $\es{C}(c, d)$ is the set of morphisms in $\es{C}$ from $c$ to $d$.

\begin{definition}
The $\bb{Q}$-module $A_\bb{Q}(\es{C})$ is called a \emph{(rational) abstract Burnside ring} if the Burnside homomorphism is an injective $\bb{Q}$-module map.
\end{definition}

If the Burnside homomorphism $\phi: A_\bb{Q}(\es{C}) \to \bb{Q}^\es{C}$ is injective, then it is an isomorphism since domain and codomain are $\bb{Q}$-vector spaces of the same dimension. As such, there is a unique ring structure on $A_\bb{Q}(\es{C})$ such that the Burnside homomorphism is a ring map. 

By construction, the Burnside homomorphism is represented by the \emph{hom-set $|\es{C}| \times |\es{C}|$-matrix} $H$ given by 
\[
H_{cd} = |\es{C}(c,d)|.
\]
This observation reduces the analysis of the Burnside homomorphism into a $\bb{Q}$-linear algebra problem.

To guarantee that $A_\bb{Q}(\es{C})$ becomes an abstract Burnside ring, we restrict to those categories $\es{C}$  which may be equipped with a factorization system. This amounts to providing a method through which we can study the hom-set matrix through standard matrix decompositions.

\begin{definition}
A category $\es{C}$ has an \emph{epi-mono factorization system} if there is a pair of wide\footnote{A \emph{wide} subcategory of a category $\es{C}$ is a subcategory containing all the objects of $\es{C}$.} subcategories $(\es{E}, \es{M})$ such that
\begin{enumerate}
    \item each morphism in $\es{E}$ (resp. $\es{M}$) is an epimorphism (resp. monomorphism); 
    \item the subcategories contain all isomorphisms;
    and,
    \item every morphism in $\es{C}$ factors uniquely (up to isomorphism) as a morphism in $\es{E}$ followed by a morphism in $\es{M}$.
\end{enumerate}
\end{definition}

If $\es{C}$ is a finite skeletal category, it follows that for each object $c \in \es{C}$
\[
\Aut(c) = \es{E}(c,c) = \es{M}(c,c),
\]
where $ \es{E}(c,d) \coloneqq \es{C}(c,d) \cap \es{E} \subseteq \sf{Epi}(c,d)$ and $\es{M}(c,d) \coloneqq \es{C}(c,d) \cap \es{M} \subseteq \sf{Mono}(c,d)$, see e.g., ~\cite[Lemma 1.3]{YOT}.

We may define $|\es{C}| \times |\es{C}|$-matrices that record the data of an epi-mono factorization system $(\es{E}, \es{M})$: the \emph{epimorphism matrix} $E$ is defined by $E_{cd} = |\es{E}(c,d)|/|\Aut(d)|$, the \emph{monomorphism matrix} $M$ is defined by $M_{cd} = |\es{M}(c,d)|/|\Aut(c)|$, and, the \emph{automorphism matrix} $A$ is defined  by $A_{cd} = |\Aut(c)|\delta_{cd}$, where $\delta_{cd}$ is the Kronecker delta. Despite being defined as rational numbers, the entries of these matrices are integral: for example, for the epimorphism matrix this follows since the automorphism group $\Aut(d)$ acts freely on $\es{E}(c,d)$ by the cancellation property of epimorphisms. The case for the monomorphism matrix is dual. Moreover, since $\Aut(c) = \es{E}(c,c) = \es{M}(c,c)$, we have that $E_{cc}=1=M_{cc}$ for each $c \in \es{C}$.

\begin{lem}[{\cite[Lemma 4.1]{YoshidaBurnside}}]\label{lem: LDU decomp}
The epimorphism, monomorphism and automorphism matrices are integral and
\[
H = EAM.
\]
\end{lem}

By~\cite[\S4.2]{YoshidaBurnside}, the epimorphism matrix is conjugate to a lower triangular matrix, the monomorphism matrix is conjugate to an upper triangular one and the automorphism matrix is diagonal, hence the decomposition of $H$ from~\cref{lem: LDU decomp} implies that $\det(H) = \prod_{c \in \es{C}} |\Aut(c)|$, which is non-zero since $\Aut(c) \neq \varnothing$ for every $c\in \es{C}$. It follows that the hom-set matrix is rationally invertible and the Burnside homomorphism is an isomorphism. 

\begin{lem}\label{cor: splitting of rational ABR}
Let $\es{C}$ be a finite skeletal category. If $\es{C}$ has an epi-mono factorization system, then the rational Burnside homomorphism
\[
\phi: A_\bb{Q}(\es{C}) \longrightarrow \bb{Q}^\es{C},
\]
is an isomorphism. In particular, $A_\bb{Q}(\es{C})$ is a rational abstract Burnside ring.
\end{lem}

\begin{rem}
The rational situation is quite different from the integral case. Integrally, Yoshida \cite[\S3.5]{YoshidaBurnside} also requires the existences of certain coequalisers in $\es{C}$. This discrepancy amounts to the fact that integrally it is not enough for the Burnside homomorphism to be injective, but rather we also require the image to be closed under multiplication.
\end{rem}

The strength in the rational Burnside ring is that the isomorphism  of~\cref{cor: splitting of rational ABR} allows us to identify idempotents through the process of M\"obius inversion, i.e., by inverting the hom-set matrix.

\begin{lem}[{\cite[Lemma 5.4]{YoshidaBurnside}}]\label{lem: idempotent}
Let $\es{C}$ be a finite skeletal category. If $\es{C}$ has an epi-mono factorization system, then the idempotent $e_d$ of $A_\bb{Q}(\es{C})$ corresponding to $d \in \es{C}$ is given by
\[
e_d = \sum_{c \in \es{C}} (H^{-1})_{cd}~c,
\]
where $H$ is the hom-set matrix. 
\end{lem}

The abstract Burnside ring of Yoshida recovers the classical Burnside ring for a finite group $G$, see~\cite[Example 3.3.(a)]{YoshidaBurnside}. We will illustrate that on a simple example.

\begin{ex}\label{ex: C6}
Let $\es{C}$ be the orbit category of the $C_6$, the cyclic group of order $6$. Since 
\[
\sf{Orb}_{C_6}(C_6/L, C_6/K) = (C_6/K)^L,
\]
we have that the hom-set matrix representing the abstract Burnside homomorphism for $A_\bb{Q}(C_6)$ is given by 
\[
H
=\begin{pmatrix}
    |(C_6/C_6)^{C_6}| & |(C_6/C_3)^{C_6}| & |(C_6/C_2)^{C_6}| & |(C_6/C_1)^{C_6}|\\
    |(C_6/C_6)^{C_3}| & |(C_6/C_3)^{C_3}| & |(C_6/C_2)^{C_3}| & |(C_6/C_1)^{C_3}| \\
    |(C_6/C_6)^{C_2}| & |(C_6/C_3)^{C_2}| & |(C_6/C_2)^{C_2}| & |(C_6/C_1)^{C_2}| \\
    |C_6/C_6| & |C_6/C_3| & |C_6/C_2| & |C_6/C_1|
\end{pmatrix}
=\begin{pmatrix}
    1 & 0 & 0 & 0\\
    1 & 2 & 0 & 0 \\
    1 & 0 & 3 & 0 \\
    1 & 2 & 3 & 6
\end{pmatrix}.
\]
The $\bb{Q}$-module $A_\bb{Q}(C_6)$ is additively generated by the orbits $C_6$, $C_6/C_2$, $C_6/C_3$ and $C_6/C_6$, so too is the classical Burnside ring and hence the Burnside homomorphism agrees with the tom Dieck’s Isomorphism~\cite[5.6.4, 5.9.13]{tomDieck}, see also~\cite[Lemma 2.11]{BarnesKedziorek}. The ring structure on $A_\bb{Q}(C_6)$ also agrees with the ring structure on the Burnside ring~\cite[Example 6.1]{BarnesKedziorek}, with multiplication given by
\[
\begin{array}{ccc}
C_6 \times C_6 = 6 \cdot C_6 &
C_6 \times C_6/C_3 = 2\cdot  C_6 &
C_6 \times C_6/C_2 = 3\cdot C_6  \\
C_6/C_3 \times C_6/C_3 = 2 \cdot C_6/C_3 &
C_6/C_2 \times C_6/C_3 = C_6   &
C_6/C_2 \times C_6/C_2 = 3 \cdot C_6/C_2,
\end{array}
\]
with $C_6/C_6$ acting as the multiplicative unit. 

Inverting the hom-set matrix $H$ gives 
\[
H^{-1}=\begin{pmatrix}
    1 & 0 & 0 & 0\\
     -\frac{1}{2} & \frac{1}{2} & 0 & 0 \\
     -\frac{1}{3} & 0 & \frac{1}{3} & 0 \\
     \frac{1}{6} & -\frac{1}{6} & -\frac{1}{6} & \frac{1}{6}
\end{pmatrix}
\]
which leads to an orthogonal decomposition
of the unit into idempotents $e_K$, for $K\leq C_6$ corresponding to projection of $\bb{Q}^4$ to the factor $C_6/K$:

\begin{table}[!ht]
    \centering
    \begin{tabular}{ll}
     $e_{C_1} = \frac{1}{6}\cdot C_6$, & $e_{C_2} = \frac{1}{3}\cdot C_6/C_2 -\frac{1}{6}\cdot C_6$, \\
     & \\
     $e_{C_3} = \frac{1}{2}\cdot C_6/C_3 - \frac{1}{6}\cdot C_6$, \quad \quad  & $e_{C_6} = C_6/C_6 - \frac{1}{2} \cdot C_6/C_3 - \frac{1}{3} \cdot C_6/C_2 + \frac{1}{6} \cdot C_6$. \\
    \end{tabular}
\end{table}

One can check that these formulae agree with the Gluck formulae~\cite[\S3]{Gluck}, see also \cite[Lemma 2.12]{BarnesKedziorek}, for the idempotents in the rational Burnside ring~\cite[Example 6.1]{BarnesKedziorek}. 
\end{ex}

We now examine how functors between finite skeletal categories interact with the abstract Burnside rings. For the remainder of this section, we assume all categories satisfy the conditions of~\cref{cor: splitting of rational ABR}. A functor $f: \es{C} \to \es{D}$ between finite skeletal categories induces a (unique) ring map
\[
f^\ast: A_\bb{Q}(\es{D}) \longrightarrow A_\bb{Q}(\es{C}), 
\]
such that
\[
\phi_\es{C}(f^\ast(d)) = (|\es{D}(f(c), d)|)_{c\in \es{C}},
\]
i.e., $f^\ast: A_\bb{Q}(\es{D}) \to A_\bb{Q}(\es{C})$ is the unique map making the diagram
\[\begin{tikzcd}
	{A_\bb{Q}(\es{D})} & {\bb{Q}^\es{D}} \\
	{A_\bb{Q}(\es{C})} & {\bb{Q}^\es{C}}
	\arrow["{\phi_\es{D}}", from=1-1, to=1-2]
	\arrow["{f^\ast}"', from=1-1, to=2-1]
	\arrow["{f^\ast}", from=1-2, to=2-2]
	\arrow["{\phi_\es{C}}"', from=2-1, to=2-2]
\end{tikzcd}\]
commute, where $f^\ast: \bb{Q}^\es{D} \to \bb{Q}^\es{C}$ is given by precomposition with $f$ upon identifying the products with the sets of functions from the sets of objects to the rationals. If the functor is an inclusion of a full subcategory with a certain condition on hom sets, then the kernel of the induced map on abstract Burnside rings is easily identified.

\begin{lem}\label{lem: kernel of map of burnside rings}
If $i: \es{C} \to \es{D}$ is an inclusion of a full subcategory such that there are no maps in $\es{D}$ from any object $c\in \es{C}$ to any object $d\in \es{D} \setminus \es{C}$, then the ring map 
\[
i^\ast : A_\bb{Q}(\es{D}) \to A_\bb{Q}(\es{C})
\]
has kernel given by the free $\bb{Q}$-module $\langle \es{D} \setminus \es{C} \rangle$  generated by the set $\es{D} \setminus \es{C}$ of objects of $\es{D}$ which are not objects of $\es{C}$.
\end{lem}
\begin{proof}
The functor $i^\ast : A_\bb{Q}(\es{D}) \to A_\bb{Q}(\es{C})$ is the identity on the objects of $\es{C}$ and zero on the objects of $\es{D}$ which are not objects of $\es{C}$, hence each element of $\es{D} \setminus \es{C}$ lies in the kernel of $i^\ast$. The proof is then complete by observing that the dimension of the kernel and the dimension of the free $\bb{Q}$-module $\langle \es{D} \setminus \es{C}\rangle$ coincide.
\end{proof}

\section{Epiorbital categories of epimorphisms}
We will now provide an example of the abstract Burnside ring machinery of~\cref{sec: ABR} by choosing the finite skeletal category to be the (skeleton of the) category $\sf{Epi}_{\leq d}$ of epimorphisms of finite sets of cardinality at most $d$. In the case $d=2$, there is an equivalence between the category of epimorphisms and the obit category $\sf{Orb}_{C_2}$ of the group $C_2$ of two elements. In general, however, the category $\sf{Epi}_{\leq d}$ is not equivalent to the orbit category of a finite group.

\begin{definition}
For each $d \geq 1$, define the \emph{rational Goodwillie Burnside ring} $A_\bb{Q}(d)$ to be the rational abstract Burnside ring associated to the category $\sf{Epi}_{\leq d}$. 
\end{definition}

This terminology follows \cite[Section 8]{AroneBarthelHeardSanders}. In~\cref{sec: goodwillie calculus} we will identify (\cref{lem: Goodwillie burnside as endo ring}) the abstract Burnside ring with the endomorphism ring of the unit of the $\infty$-category of rational $d$-excisive functors. 

\begin{ex}
The category $\sf{Epi}_{\leq 2}$ (or equivalently $\sf{Orb}_{C_2}$) may be displayed pictorially as
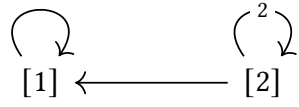
\[\begin{tikzcd}
	{[1]} && {[2]}
	\arrow[from=1-1, to=1-1, loop, in=55, out=125, distance=10mm]
	\arrow[from=1-3, to=1-1]
	\arrow["2"{description}, from=1-3, to=1-3, loop, in=55, out=125, distance=10mm]
\end{tikzcd}\]
where a number attached to an arrow indicates the cardinality of the set of arrows of this form. No number indicates a single arrow. In this case we have that the hom-set matrix and its inverse are given by
\[
H =
\begin{pmatrix}
    1 & 0 \\
    1 & 2 \\
\end{pmatrix},\quad \quad 
H^{-1}  =
\begin{pmatrix}
    1 & 0 \\
    -\frac{1}{2} & \frac{1}{2} \\
\end{pmatrix}.
\]
We are predominantly interested in understanding the idempotents of the rational Goodwillie Burnside ring $A_\bb{Q}(d)$. Recall from~\cref{lem: idempotent} that these idempotents  in general are given by
\[
e_d = \sum_{c \in \es{C}} (H^{-1})_{cd}~c,
\]
hence, for the rational Goodwillie-Burnside ring $A_\bb{Q}(2)$ the idempotents are given by
\begin{table}[ht]
    \centering
    \begin{tabular}{cc}
      $e_2 = \frac{1}{2}\cdot[2]$, \quad \quad & $e_1 = [1] -\frac{1}{2}\cdot[2]$. \\
    \end{tabular}
\end{table}
\end{ex}

\begin{ex}
The category $\sf{Epi}_{\leq 3}$ may be displayed pictorially as
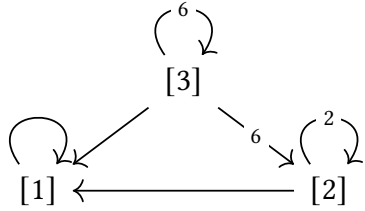
\[\begin{tikzcd}
	& {[3]} \\
	{[1]} && {[2]}
	\arrow["6"{description}, from=1-2, to=1-2, loop, in=55, out=125, distance=10mm]
	\arrow[from=1-2, to=2-1]
	\arrow["6"{description}, from=1-2, to=2-3]
	\arrow[from=2-1, to=2-1, loop, in=55, out=125, distance=10mm]
	\arrow[from=2-3, to=2-1]
	\arrow["2"{description}, from=2-3, to=2-3, loop, in=55, out=125, distance=10mm]
\end{tikzcd}\] 
In this case, the hom-set matrix and its inverse are given by
\[
H =
\begin{pmatrix}
    1 & 0 & 0 \\
    1 & 2 & 0 \\
    1 & 6 & 6 \\
\end{pmatrix}, \quad \quad 
H^{-1}  =
\begin{pmatrix}
    1 & 0 & 0 \\
    -\frac{1}{2} & \frac{1}{2} &0  \\
    \frac{1}{3} & -\frac{1}{2} & \frac{1}{6} 
\end{pmatrix}.
\]
Applying~\cref{lem: idempotent}, the idempotents in the rational Goodwillie-Burnside ring $A_\bb{Q}(3)$ are given by

\begin{table}[ht]
    \centering
    \begin{tabular}{ccc}
      $e_3= \frac{1}{6}\cdot[3]$, \quad \quad & $e_2 = \frac{1}{2}\cdot[2] - \frac{1}{2}\cdot[3]$, \quad \quad & $e_1 = [1] -\frac{1}{2}\cdot[2] + \frac{1}{3}\cdot[3]$. \\
    \end{tabular}
\end{table}
\end{ex}

Our final low dimensional example highlights the passage between the Goodwillie-Burnside ring $A(d-1)$ and the Goodwillie Burnside ring $A(d)$ through the inclusion of $\sf{Epi}_{\leq d-1}$ as a full subcategory of $\sf{Epi}_{\leq d}$, which satisfies the conditions of~\cref{lem: kernel of map of burnside rings}.

\begin{ex}
The category $\sf{Epi}_{\leq 4}$ may be displayed pictorially as
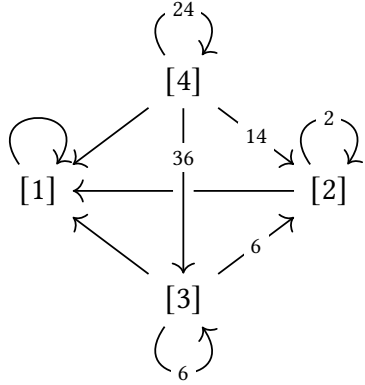
\[\begin{tikzcd}
	& {[4]} \\
	{[1]} && {[2]} \\
	& {[3]}
	\arrow["24"{description}, from=1-2, to=1-2, loop, in=55, out=125, distance=10mm]
	\arrow[from=1-2, to=2-1]
	\arrow["14"{description}, from=1-2, to=2-3]
	\arrow[from=2-1, to=2-1, loop, in=55, out=125, distance=10mm]
	\arrow[from=2-3, to=2-1]
	\arrow["2"{description}, from=2-3, to=2-3, loop, in=55, out=125, distance=10mm]
	\arrow[from=3-2, to=2-1]
	\arrow["6"{description}, from=3-2, to=2-3]
	\arrow["6"{description}, from=3-2, to=3-2, loop, in=305, out=235, distance=10mm]
    \arrow["36"{description, pos=0.3}, from=1-2, to=3-2, crossing over]
\end{tikzcd}\]
In this case we have that the hom-set matrix and its inverse are given by
\[
H =
\begin{pmatrix}
    1 & 0 & 0 & 0 \\
    1 & 2 & 0 & 0  \\
    1 & 6 & 6 &0 \\
    1 & 14 & 36 &24 \\
\end{pmatrix}, \quad \quad 
H^{-1}  =
\begin{pmatrix}
    1 & 0 & 0 &0 \\
    -\frac{1}{2} & \frac{1}{2} &0 & 0 \\
    \frac{1}{3} & -\frac{1}{2} & \frac{1}{6} &0 \\
    -\frac{1}{4} & \frac{11}{24} & -\frac{1}{4} & \frac{1}{24}
\end{pmatrix}.
\]
The rational idempotents of $A_\bb{Q}(4)$ are given by

\begin{table}[!ht]
    \centering
    \begin{tabular}{ll}
     $e_4 = \frac{1}{24}\cdot[4]$, & $e_3 = \frac{1}{6}\cdot[3]-\frac{1}{4}\cdot[4]$, \\
     & \\
     $e_2 = \frac{1}{2}\cdot[2] - \frac{1}{2}\cdot[3] + \frac{11}{24}\cdot[4]$, \quad \quad  & $e_1 = [1] -\frac{1}{2}\cdot[2] + \frac{1}{3}\cdot[3] - \frac{1}{4}\cdot[4]$. \\
    \end{tabular}
\end{table}
The inclusion of the full subcategory $\sf{Epi}_{\leq 3} \hookrightarrow \sf{Epi}_{\leq 4}$ induces a ring map $A_\bb{Q}(4) \to A_\bb{Q}(3)$. This inclusion of the full subcategory satisfies the conditions of~\cref{lem: kernel of map of burnside rings}, so the kernel is generated by $[4] \in A_\bb{Q}(4)$. By our computation of the idempotents above, this is equivalent to the kernel being generated by the ``top idempotent'' $e_4$.
\end{ex}

We now summarise the previous examples in the general case of the category of finite sets of cardinality at most $d$ and epimorphisms. 

\begin{ex}\label{ex: general d}
The hom-set matrix for $\sf{Epi}_{\leq d}$ and its inverse are given by
\[
H = \begin{pmatrix}
1 & 0 & 0 & \cdots & 0 \\
1 & 2 & 0 & \cdots & 0  \\
1 & 6 & 3! & \cdots & 0\\
\vdots&  \vdots & \vdots & \ddots &  0\\
1& |\sf{surj}(d,2)| & |\sf{surj}(d,3)|  & \cdots & d!\\
\end{pmatrix}, \quad 
H^{-1} =  \begin{pmatrix}
1 & 0 & 0 & \cdots & 0 \\
-\frac{1}{2} & \frac{1}{2} & 0 & \cdots & 0  \\
\frac{1}{3} & -\frac{1}{2} & \frac{1}{3!} & \cdots & 0\\
\vdots&  \vdots & \vdots & \ddots &  0\\
\frac{1}{d!} s(d,1) & \frac{1}{d!} s(d,2) & \frac{1}{d!} s(d,3)  & \cdots & \frac{1}{d!}\\
\end{pmatrix}
\] 
where $|\sf{surj}(i,j)|$ is the cardinality of the set of surjections from a set with $i$ elements to a set with $j$ elements and $s(i,j)$ is the (signed) Stirling number of the first kind. Observe that the number of surjections is given by, $|\sf{surj}(i,j)| = j!S(i,j)$, where $S(i,j)$ is the Stirling number of the second kind. In the computation of $H^{-1}$ we make use of the fact that the inverse of $H$ exists since none of the diagonal elements $H_{k,k} = |\sf{surj}(k,k)|=k!$ are zero. Since $H$ is lower triangular, we have that $H^{-1}$ is also lower triangular, and that the diagonal elements of $H^{-1}$ are the reciprocal of the diagonal elements of $H$. The terms away from the leading diagonal may be computed iteratively, simply using the equation $H^{-1}H = I$. In particular, we see that the ``top idempotent'' corresponding to $[d] \in \sf{Epi}_{\leq d}$ is given by 
\[
e_d = \frac{1}{d!}\cdot[d]
\]
and \cref{lem: idempotent} along with our description of $H^{-1}$ gives the other idempotents.  
\end{ex}

\begin{ex}\label{ex:inclusion of subcategory epi}
    The inclusion of the full subcategory $i_{d-1}: \sf{Epi}_{\leq d-1} \hookrightarrow \sf{Epi}_{\leq d}$ induces a ring map $(i_{d-1})^\ast: A_\bb{Q}(d) \to A_\bb{Q}(d-1)$. This inclusion of the full subcategory satisfies the conditions of~\cref{lem: kernel of map of burnside rings}, so the kernel is generated by $[d] \in A_\bb{Q}(d)$. By our computation of the idempotents above, this is equivalent to the kernel being generated by the ``top idempotent'' $e_d$.
\end{ex}

\section{Excisive functors}
In the next sections, we will apply the abstract Burnside ring machinery to the study of excisive functors in Goodwillie calculus. 
We begin by recalling some necessary background.

An $n$-cube $\es{X}$ in an $\infty$-category $\es{C}$ is a functor $\es{X}: \es{P}(n) \to \es{C}$, where $\es{P}(n)$ is the poset of subsets of the set of $n$ elements. Assuming that $\es{C}$ is bicomplete, an $n$-cube is \emph{strongly cocartesian} if every face is a pushout, \emph{cartesian} if it is a limit diagram and \emph{cocartesian} if it is a colimit diagram. 

A functor $F: \s^\omega \to \s$ is \emph{$d$-excisive} if it sends strongly cocartesian $(d+1)$-cubes to cartesian cubes. One should view this as a weak form of excision, i.e., a $1$-excisive functor sends pushout diagrams to pullback diagrams.  We will say that $F$ is \emph{reduced} if $F(\ast) \simeq \ast$. Denote by $\exc{d}(\s^\omega,\s)$ the $\infty$-category of reduced $d$-excisive endofunctors of spectra\footnote{For technical reasons, we must restrict the domain to the full subcategory of compact objects but we shall not belabour this point. We will also only consider reduced functors, hence drop the reduced terminology and refer to the objects of $\exc{d}(\s^\omega,\s)$ as $d$-excisive functors. The reduced assumption is not limiting since any functor $F: \s \to \s$ is ``reduced up to a constant'', i.e., decomposes as $F(X) \simeq \overline{F}(X) \times F(\ast)$, where $\overline{F}$ is reduced.}. 

To any $F: \s^\omega \to \s$, Goodwillie calculus assigns a sequence of approximations 
\[\begin{tikzcd}
	&& F \\
	\cdots & {P_dF} & \cdots & {P_1F} & {P_0F = F(\ast)}
	\arrow[from=2-4, to=2-5]
	\arrow[from=2-2, to=2-3]
	\arrow[from=2-3, to=2-4]
	\arrow[bend right=30, from=1-3, to=2-2]
	\arrow[from=2-1, to=2-2]
	\arrow[bend left=30, from=1-3, to=2-4]
	\arrow[bend left=20, from=1-3, to=2-5]
\end{tikzcd}\]
in which $P_dF$ is the universal $d$-excisive approximation to $F$. 
A comprehensive account of the homotopy theory of excisive functors of spectra is given in~\cite[\S2]{AroneBarthelHeardSanders}. The key technical input about the $\infty$-category of $d$-excisive functors is that the category of $d$-excisive functors is closed under Day convolution of functors.

\begin{thm}[{\cite[Theorem 2.38]{AroneBarthelHeardSanders}}]
The $\infty$-category $\exc{d}(\s^\omega, \s)$ of $d$-excisive functors is a presentably symmetric monoidal stable $\infty$-category with monoidal structure given by Day convolution and monoidal unit $P_dh_\bb{S} = P_d\Sigma^\infty\Map(\bb{S},-) \simeq P_d\Sigma^\infty\Omega^\infty$.
\end{thm}

Since rationalisation is a smashing localization, we obtain the following corollary, which summarises the key technical information about the category of rational $d$-excisive functors.

\begin{cor}\label{cor: local excisive category}
The $\infty$-category $\exc{d}(\s^\omega,\s)_\bb{Q}$ of rational $d$-excisive functors is a presentably symmetric monoidal stable $\infty$-category with monoidal structure given by Day convolution and monoidal unit $L_\bb{Q}P_dh_\bb{S}$.
\end{cor}

A functor $F$ is \emph{$d$-homogeneous} if it is $d$-excisive and $P_{d-1}F\simeq \ast$. We will denote the $\infty$-category of $d$-homogeneous functors by $\homog{d}(\s^\omega,\s)$. Since the $\infty$-category of $d$-excisive functors is closed under fibres, and a $(d-1)$-excisive functor is $d$-excisive, the $d$-th layer of the tower,
\[
D_dF = \fib(P_dF \longrightarrow P_{d-1}F),
\]
is $d$-homogeneous, and there is an adjunction
\[
\adjunction{\iota}{\homog{d}(\s^\omega,\s)}{\exc{d}(\s^\omega,\s)}{D_d},
\]
realising the $\infty$-category of $d$-homogeneous functors as a coreflexive sub-$\infty$-category of the $\infty$-category of $d$-excisive functors. It follows that a $d$-excisive functor $F$ has a Goodwillie tower of the form
\[\begin{tikzcd}
	F & {P_{d-1}F} & \cdots & {P_1F} \\
	{D_dF} & {D_{d-1}F} && {D_1F}.
	\arrow[from=1-1, to=1-2]
	\arrow[from=1-2, to=1-3]
	\arrow[from=1-3, to=1-4]
	\arrow[from=2-1, to=1-1]
	\arrow[from=2-2, to=1-2]
	\arrow["\simeq"', from=2-4, to=1-4]
\end{tikzcd}\]
 
\section{The rational Goodwillie-Burnside ring}\label{sec: goodwillie calculus}
Our attention now turns to the role of the abstract Burnside ring $A_\bb{Q}(d)$ in Goodwillie calculus, a comprehensive account of which is given in~\cite[\S8-9]{AroneBarthelHeardSanders}. We work rationally, as this allows us to simplify and streamline certain arguments, but integral versions of many of the statements in this section may be found in \emph{loc. cit.}.

The Goodwillie derivatives assemble into a left adjoint 
\[
\partial_\ast: \exc{d}(\s^\omega,\s) \longrightarrow \sf{SymSeq},
\]
between the $\infty$-category of $d$-excisive functors and the $\infty$-category of symmetric sequences~\cite[Proposition 4.5]{ACClassification}. This functor factors through $d$-truncated symmetric sequences since $\partial_iP_d \simeq 0$ for $i>d$. The mapping spectra in symmetric sequences decompose as a product, so there is an induced map on mapping spectra
\[
\partial_\ast : \Hom_{\exc{d}(\s^\omega,\s)}(P_dh_\bb{S}, P_dh_\bb{S}) \longrightarrow \prod_{1 \leq i \leq d} \Hom_{\s}(\partial_i(P_dh_\bb{S}), \partial_i(P_dh_\bb{S})),
\]
which induces a ring map 
\[
\theta_d : \pi_0\sf{End}_{\exc{d}(\s^\omega,\s)}(P_dh_\bb{S}) \otimes \bb{Q} \longrightarrow \prod_{1 \leq i \leq d} \pi_0\Hom_{\s}(\partial_i(P_dh_\bb{S}), \partial_i(P_dh_\bb{S})) \otimes \bb{Q} \cong \bb{Q}^d
\]
on the level of rational homotopy categories, where the last isomorphism follows from Arone and Ching's computation~\cite[Lemma 5.10]{ACClassification} of the derivatives of representables as $\partial_i (P_dh_\bb{S}) \simeq \bb{S}$ with trivial $\Sigma_i$-action for all $1 \leq i \leq d$.

To ease notation, we will denote the endomorphism ring $\pi_0\Hom_{\s}(P_dh_\bb{S},P_dh_\bb{S}) \otimes \bb{Q}$ by $R_\bb{Q}(d)$. We will now show that Goodwillie-Burnside ring $A_\bb{Q}(d)$ is isomorphic to this endomorphism ring $R_\bb{Q}(d)$, by an intricate analysis of how the calculus and the $\bb{Q}$-linear algebra interact. 

\begin{lem}\label{lem: endomorphism vs yoshida}
There is an isomorphism
\[
R_\bb{Q}(d) \cong A_{\bb{Q}}(d)
\]
between the endomorphism ring of the monoidal unit in $\exc{d}(\s^\omega,\s)_\bb{Q}$ and the rational Goodwillie-Burnside ring.
\end{lem}
\begin{proof}
There is a zigzag of ring maps
\[
R_\bb{Q}(d) \xrightarrow{ \ \theta_d \ } \bb{Q}^d \xleftarrow{\ \phi \ } A_\bb{Q}(d),
\]
in which the right-most map is an isomorphism by~\cref{cor: splitting of rational ABR}, hence it suffices to show that $\theta_d: R_\bb{Q}(d) \to \bb{Q}^d$ is an isomorphism. 

By an elementary degree computation~\cite[Lemma 9.20]{AroneBarthelHeardSanders}, there are 
elements $[\lambda_m] \in R_\bb{Q}(d)$, $1 \leq m \leq d$, such that the 
$i$th component of $\theta_d [\lambda_m]$
is $|\sf{surj}(i,m)|$, see~\cref{ex: general d}. The same example shows that the vectors $\theta_d[\lambda_1], \dots, \theta_d[\lambda_d]$ are linearly independent, hence $\theta_d$ is surjective for all $d$.

To see that $\theta_d$ is injective, and hence an isomorphism, we work by induction over $d$. In the initial case of $d=1$,  we have equivalences
\[
\Hom_{\exc{1}(\s^\omega,\s)}(P_1h_\bb{S}, P_1h_\bb{S}) 
\xrightarrow{\partial_1}  
\Hom_{\s}(\partial_1(P_1h_\bb{S}), \partial_1(P_1h_\bb{S})),
\simeq 
\Hom_{\s}(\bb{S}, \bb{S}),
\]
by the classification of homogeneous functors, since $D_1h_\bb{S} \simeq P_1h_\bb{S}$. Taking $\pi_0$ gives the required ring isomorphism $\theta_1 : R_\bb{Q}(1) \to \bb{Q}$.

For the induction step, assume that $\theta_{d-1}$ is an isomorphism. Since $\partial_iP_{d-1} \simeq \partial_i$ for $1 \leq i \leq d-1$ and $\partial_dP_{d-1}\simeq 0$, there is a commutative diagram 
\[\begin{tikzcd}
	{R_\bb{Q}(d)} & {\bb{Q}^d} \\
	{R_\bb{Q}(d-1)} & {\bb{Q}^{d-1}}
	\arrow["{\theta_d}", from=1-1, to=1-2]
	\arrow["{\pi_0(P_{d-1}) \otimes \bb{Q}}"', from=1-1, to=2-1]
	\arrow["{(i_{d-1})^\ast}", from=1-2, to=2-2]
	\arrow["{\theta_{d-1}}", from=2-1, to=2-2]
\end{tikzcd}\]
where $(i_{d-1})^\ast : \bb{Q}^d \to \bb{Q}^{d-1}$ is projection onto the first $(d-1)$ factors. 
The above square induces a map
\[
\bar{\theta}_d: \ker(\pi_0(P_{d-1})\otimes \bb{Q}) \longrightarrow \ker((i_{d-1})^\ast),
\]
on vertical kernels. If $\bar{\theta}_d$ is an isomorphism, a diagram chase will show that $\theta_d : R_\bb{Q}(d) \to \bb{Q}^d$ is
injective and hence an isomorphism as desired. 

The element $[\lambda_d]$ is in the kernel of $\pi_0(P_{d-1}) \otimes \bb{Q}$, hence $(\theta_{d-1} \circ \pi_0(P_{d-1}) \otimes \bb{Q})[\lambda_d] =0$, and so, by commutativity of the above square, $\theta_d[\lambda_d]$ is in the kernel of $(i_{d-1})^\ast$. Since $\theta_d[\lambda_d]= (0, \dots, 0, d!)$ the map $\bar{\theta}_d$ is non-zero and 
it is therefore surjective as the projection onto the first $(d-1)$ factors, $(i_{d-1})^\ast : \bb{Q}^d \to \bb{Q}^{d-1}$ has $\ker((i_{d-1})^\ast) \cong \bb{Q}$. 

To show that the map on kernels is an isomorphism, we will show that both kernels have the same dimension. By the Yoneda lemma and the universal property of the excisive approximations the exact sequence
\[
\ker(\pi_0(P_{d-1})\otimes \bb{Q}) \longrightarrow R_\bb{Q}(d) \longrightarrow R_\bb{Q}(d-1),
\]
may be identified with the exact sequence
\[
\pi_0(D_dh_\bb{S}(\bb{S}))\otimes \bb{Q} \longrightarrow \pi_0(P_dh_\bb{S}(\bb{S}))\otimes \bb{Q} \longrightarrow \pi_0(P_{d-1}h_\bb{S}(\bb{S}))\otimes \bb{Q},
\]
induced by the $d$-homogeneous fibre sequence of $h_\bb{S}$ evaluated at $\bb{S}$. Computing $d$-homogenous layer of $h_\bb{S}$, we see that
\[
D_dh_\bb{S}(\bb{S}) \simeq (\partial_dh_\bb{S} \otimes \bb{S}^{\otimes d})_{h\Sigma_d} \simeq (\bb{S} \otimes \bb{S}^{\otimes d})_{h\Sigma_d} \simeq \bb{S}^{\otimes d}_{h\Sigma_d} \simeq \Sigma^\infty_+B\Sigma_d.
\]
It follows that $\pi_0(D_dh_\bb{S}(\bb{S})) \otimes \bb{Q} \cong \pi_0(\Sigma^\infty_+B\Sigma_d) \otimes \bb{Q} \cong \bb{Q}$, and  
hence there is an isomorphism $\ker(\pi_0(P_{d-1})\otimes \bb{Q}) \cong \bb{Q}$. It follows that the domain and codomain of $\bar{\theta}_d$ are isomorphic to $\bb{Q}$, and hence $\bar{\theta}_d$ is an isomorphism. 
\end{proof}

The identification of the Goodwillie-Burnside ring with the endomorphism ring is compatible with the Goodwillie tower and the rational idempotents in the following sense.

\begin{lem}\label{lem: Goodwillie burnside as endo ring}
There is a commutative diagram
\[\begin{tikzcd}
	{R_\bb{Q}(d)} & {\bb{Q}^d} & {A_\bb{Q}(d)} \\
	{R_\bb{Q}(d-1)} & {\bb{Q}^{d-1}} & {A_\bb{Q}(d-1)}
	\arrow["{\theta_d}", from=1-1, to=1-2]
	\arrow["{\pi_0(P_{d-1}) \otimes \bb{Q}}"', from=1-1, to=2-1]
	\arrow["{(i_{d-1})^\ast}"', from=1-2, to=2-2]
	\arrow["{\phi_d}"', from=1-3, to=1-2]
	\arrow["{(i_{d-1})^\ast}", from=1-3, to=2-3]
	\arrow["{\theta_{d-1}}"', from=2-1, to=2-2]
	\arrow["{\phi_{d-1}}", from=2-3, to=2-2]
\end{tikzcd}\]
in which each horizontal map is an isomorphism and the right-hand map is discussed in \cref{ex:inclusion of subcategory epi}. In particular, the kernel of $\pi_0(P_{d-1}) \otimes \bb{Q}$ is generated by the top idempotent in the rational Goodwillie-Burnside ring.
\end{lem}
\begin{proof}
We have already seen that the left square commutes, and the right square commutes, by definition of the map on the level of abstract Burnside rings. The claim about the rational kernel follows from combining~\cref{lem: kernel of map of burnside rings} with~\cref{lem: endomorphism vs yoshida} to see that the induced maps on the level of vertical kernels in the above diagram are all isomorphisms and the observation of~\cref{ex: general d} that the idempotent $e_d$ is a rational multiple of $[d]$. 
\end{proof}

\section{Splitting rational excisive functors}

The Goodwillie-Burnside ring $A_\bb{Q}(\sf{Epi}_{\leq d})$ associated to the category $\sf{Epi}_{\leq d}$ is equivalent to the endomorphism ring of the unit (\cref{lem: endomorphism vs yoshida}), so the splitting of the Goodwillie-Burnside ring (\cref{cor: splitting of rational ABR}) will induce the required splitting of the category of rational $d$-excisive functors (\cref{thm: splitting of exc}).

The idempotents in the Goodwillie-Burnside ring lift to idempotents in the $\infty$-category of $d$-excisive functors by applying~\cite[Lemma 1.2.4.6, Warning 1.2.4.8]{HA} and observing that the idempotents in the Goodwillie-Burnside ring are also idempotents in the homotopy category of $d$-excisive functors (\cref{lem: Goodwillie burnside as endo ring}).

Given an idempotent $e$ in a stable $\infty$-category $\es{C}$, the $\infty$-categorical enhancement of work of the first author~\cite[Theorem 4.4]{Barnes} provides a splitting of $\es{C}$ into orthogonal pieces
\[
\es{C} \simeq L_{e}\es{C} \times L_{1-e} \es{C},
\]
by localizing with respect to the action of the idempotent $e$ and the orthogonal idempotent $1-e$. Since the $\infty$-category of $d$-excisive functors is equivalent to modules over the monoidal unit~\cite[Proposition 2.35]{AroneBarthelHeardSanders}, the orthogonal idempotent splitting corresponding to $e_d$ may be equivalently phrased as
\[
\exc{d}(\s^\omega,\s)_\bb{Q} \simeq \sf{Mod}_\bb{Q}(P_dh_\bb{S}) \simeq \sf{Mod}_\bb{Q}(P_dh_\bb{S}[e_d^{-1}]) \times \sf{Mod}_\bb{Q}(P_dh_\bb{S}[(1-e_d)^{-1}]),
\]
see e.g.,~\cite{Mathew, May}, where $\sf{Mod}_\bb{Q}(-)$ is shorthand notation for the $\infty$-category of modules in the rational functor category $\Fun_\ast(\s^\omega,\s)_\bb{Q}$.

Consider the fibre sequence
\begin{equation}\label{eqn: fib seq w/ Pd}
D_d(P_dh_\bb{S}) \longrightarrow P_dh_\bb{S} \longrightarrow P_{d-1}(P_dh_\bb{S})
\end{equation}
defining the $d$-homogeneous layer of the monoidal unit of the $\infty$-category of $d$-excisive functors. Given a self map $e: P_dh_\bb{S} \to P_dh_\bb{S}$, we obtain a commutative diagram
\[\begin{tikzcd}
	{D_d(P_dh_\bb{S})} & {P_dh_\bb{S}} & { P_{d-1}(P_dh_\bb{S})} \\
	{D_d(P_dh_\bb{S})} & {P_dh_\bb{S}} & { P_{d-1}(P_dh_\bb{S})}
	\arrow[from=1-1, to=1-2]
	\arrow["{D_d(e)}"', from=1-1, to=2-1]
	\arrow[from=1-2, to=1-3]
	\arrow["e", from=1-2, to=2-2]
	\arrow["{P_{d-1}(e)}", from=1-3, to=2-3]
	\arrow[from=2-1, to=2-2]
	\arrow[from=2-2, to=2-3]
\end{tikzcd}\]
since (leaving inclusions implicit) the left horizontal maps are the counit of the adjunction between $d$-homogeneous functors and $d$-excisive functors, while the right horizontal maps are the unit of the adjunction between $d$-excisive functors and $(d-1)$-excisive functors. Hence, any idempotent $e$ of $P_dh_\bb{S}$ acts by $D_d(e)$ on the $d$-homogeneous layer and $P_{d-1}(e)$ on the $(d-1)$-excisive approximation. The fibre sequence~\eqref{eqn: fib seq w/ Pd} may be identified with the fibre sequence
\begin{equation} \label{eqn: fib seq w/out Pd}
D_dh_\bb{S} \longrightarrow P_dh_\bb{S} \longrightarrow P_{d-1}h_\bb{S}, 
\end{equation}
since $P_{d-1}P_d\simeq P_{d-1}$ and $D_dP_d\simeq D_d$. Through the equivalence between~\eqref{eqn: fib seq w/ Pd} and~\eqref{eqn: fib seq w/out Pd} an idempotent $e$ of $P_{d}h_\bb{S}$ acts by $D_d(e)$ on the $d$-homogeneous layer and $P_{d-1}(e)$ on the $(d-1)$-excisive approximation in~\eqref{eqn: fib seq w/out Pd}.

We now examine how the idempotent $e_d \in A_\bb{Q}(d)$ acts on the fibre sequence~\eqref{eqn: fib seq w/out Pd} with the goal of showing that this fibre sequence induces the desired splitting. To ease notation, we will use $F[e^{-1}]$ for the inversion of the action of the idempotent $e$ on $F$, where $F$ is any of the three functors from~\eqref{eqn: fib seq w/out Pd}.

\begin{lem}\label{lem: ed on Pd-1}
The idempotent $e_d \in A_\bb{Q}(d)$ acts on $P_{d-1}h_\bb{S}$ by zero, i.e., $P_{d-1}h_\bb{S}[e_d^{-1}]\simeq 0$. In particular, $P_{d-1}h_\bb{S}$ is $e_d$-acyclic and hence $(1-e_d)$-local.
\end{lem}

\begin{proof}
The action of the idempotent $e_d$ on $P_{d-1}h_\bb{S}$ is given by the image of $e_d$ under the map
\[
P_{d-1}: \Hom_{\exc{d}(\s^\omega,\s)}(P_dh_\bb{S}, P_dh_\bb{S}) \longrightarrow \Hom_{\exc{d}(\s^\omega,\s)}(P_{d-1}h_\bb{S}, P_{d-1}h_\bb{S}).
\]
By adjunction there is an equivalence
\[
\Hom_{\exc{d}(\s^\omega,\s)}(P_{d-1}h_\bb{S}, P_{d-1}h_\bb{S}) \simeq \Hom_{\exc{d-1}(\s^\omega,\s)}(P_{d-1}h_\bb{S}, P_{d-1}h_\bb{S}),
\]
and applying $\pi_0$ to the composite gives the map
\[
R_\bb{Q}(d) \xrightarrow{ \ \pi_0  (P_{d-1})\otimes \bb{Q} \ }  R_\bb{Q}(d-1).
\]
By~\cref{lem: Goodwillie burnside as endo ring} this is isomorphic to
\[
A_\bb{Q}(d) \xrightarrow{ \ {(i_{d-1})^\ast} \ } A_\bb{Q}(d-1),
\]
which sends $e_d$ to zero. The result follows from the observation that the action of an idempotent is determined by its action on the homotopy category,~\cite[Lemma 1.2.4.6, Warning 1.2.4.8]{HA}.
\end{proof}

\begin{lem}\label{lem: ed on Dd}
The idempotent $e_d \in A_\bb{Q}(d)$ acts on $D_dh_\bb{S}$ by the identity, i.e., $D_dh_\bb{S}[e_d^{-1}] \simeq D_dh_\bb{S}$. In particular, $D_dh_\bb{S}$ is $e_d$-local and hence $(1-e_d)$-acyclic.
\end{lem}
\begin{proof}
It suffices to show that the image of the idempotent $e_d$ under the map
\[
\Hom_{\exc{d}(\s^\omega,\s)}(P_dh_\bb{S}, P_dh_\bb{S}) \xrightarrow{ \ D_d \ } \Hom_{\exc{d}(\s^\omega,\s)}(D_dh_\bb{S}, D_dh_\bb{S}),
\]
is the identity map. 
By adjunction and the classification of homogeneous functors, there are equivalences
\[
\Hom_{\exc{d}(\s^\omega,\s)}(D_dh_\bb{S}, D_dh_\bb{S}) \simeq \Hom_{\homog{d}(\s^\omega,\s)}(D_dh_\bb{S}, D_dh_\bb{S}) \simeq \Hom_{\s}(\partial_dh_\bb{S}, \partial_dh_\bb{S}),
\]
such that the diagram
\[\begin{tikzcd}
	{\Hom_{\exc{d}(\s^\omega,\s)}(P_dh_\bb{S}, P_dh_\bb{S}) } & {\Hom_{\exc{d}(\s^\omega,\s)}(D_dh_\bb{S}, D_dh_\bb{S})} \\
	{\Hom_{\s}(\partial_dh_\bb{S}, \partial_dh_\bb{S})} & {\Hom_{\s}(\partial_dh_\bb{S}, \partial_dh_\bb{S})}
	\arrow["{D_d}", from=1-1, to=1-2]
	\arrow["{\partial_d}"', from=1-1, to=2-1]
	\arrow["\simeq", from=1-2, to=2-2]
	\arrow[Rightarrow, no head, from=2-1, to=2-2]
\end{tikzcd}\]
commutes. On the level of homotopy categories, this diagram becomes 
a commutative square 
\[\begin{tikzcd}
	{A_\bb{Q}(d)} & {\pi_0\sf{End}(D_dh_\bb{S})\otimes \bb{Q}} \\
	{\bb{Q}} & {\bb{Q}}
	\arrow["{\pi_0(D_d)}", from=1-1, to=1-2]
	\arrow["{\pi_0(\partial_d)}"', from=1-1, to=2-1]
	\arrow["\simeq", from=1-2, to=2-2]
	\arrow[Rightarrow, no head, from=2-1, to=2-2]
\end{tikzcd}\]
of rings. As $\pi_0(\partial_d)(e_d)$ is $1$ in $\bb{Q}$, it follows that 
$\pi_0(D_d)(e_d)$ is the identity map of $D_dh_\bb{S}$.
\end{proof}

Combining the previous lemmas, we obtain a complete understanding of the idempotent actions on the fibre sequence defining the $d$-homogeneous layer of the unit.

\begin{prop}\label{prop: homog fibre splits}
The fibre sequence,
\[
D_dh_\bb{S} \longrightarrow P_dh_\bb{S} \longrightarrow P_{d-1}h_\bb{S},
\]
splits rationally. Moreover,
\[
P_{d}h_\bb{S}[e_d^{-1}]  \simeq D_dh_\bb{S}
\qquad
P_dh_\bb{S}[(1-e_d)^{-1}] \simeq P_{d-1}h_\bb{S}.
\]
\end{prop}
\begin{proof}
We may split $P_{d}h_\bb{S} $ using the idempotents
\[
P_{d}h_\bb{S} \simeq P_dh_\bb{S}[e_d^{-1}] \times P_dh_\bb{S}[(1-e_d)^{-1}].
\]
Consider the first term of this splitting and the induced fibre sequence 
\[
D_dh_\bb{S}[e_d^{-1}] \longrightarrow 
P_dh_\bb{S}[e_d^{-1}] \longrightarrow 
P_{d-1}h_\bb{S}[e_d^{-1}], 
\]
By~\cref{lem: ed on Pd-1} we know the last term is zero, so the
first map is an equivalence. By~\cref{lem: ed on Dd} we know that
$D_dh_\bb{S}[e_d^{-1}] \simeq D_dh_\bb{S}$. 

A similar argument for the other term of the splitting shows that
\[
P_dh_\bb{S}[(1-e_d)^{-1}] \simeq 
P_{d-1}h_\bb{S}[(1-e_d)^{-1}] \simeq
P_{d-1}h_\bb{S}. \qedhere
\]
\end{proof}

We can now apply our knowledge of how the idempotent acts on the homogeneous fibre sequence~\eqref{eqn: fib seq w/out Pd} of the monoidal unit to identify the pieces of the splitting on the level of $\infty$-categories.

\begin{thm}\label{thm: splitting of exc}
There is an equivalence
\[
\exc{d}(\s^\omega,\s)_\bb{Q} \simeq \prod_{1 \leq i \leq d} \homog{i}(\s^\omega,\s)_\bb{Q},
\]
of $\infty$-categories.
\end{thm}
\begin{proof}
By~\cref{prop: homog fibre splits} the idempotent splitting of the $\infty$-category of $d$-excisive functors is given by
\begin{align*}
\exc{d}(\s^\omega,\s)_\bb{Q} 
&\simeq \sf{Mod}_\bb{Q}(P_dh_\bb{S}[e_d^{-1}]) \times \sf{Mod}_\bb{Q}(P_dh_\bb{S}[(1-e_d)^{-1}]) \\
&\simeq \sf{Mod}_\bb{Q}(D_dh_\bb{S}) \times \sf{Mod}_\bb{Q}(P_{d-1}h_\bb{S}) \\
&\simeq \sf{Mod}_\bb{Q}(D_dh_\bb{S}) \times \exc{d-1}(\s^\omega,\s)_\bb{Q},
\end{align*}
since $D_dh_\bb{S}$ inherits the structure of a commutative algebra object in $\Fun_\ast(\s^\omega,\s)_\bb{Q}$ through the equivalence $D_dh_\bb{S} \simeq P_dh_\bb{S}[e_d^{-1}]$. 

Through the equivalent descriptions of the splitting of $\exc{d}(\s^\omega,\s)_\bb{Q}$ the projection onto the second factor corresponds to the functor 
\[
P_{d-1}: \exc{d}(\s^\omega,\s)_\bb{Q} \longrightarrow \exc{d-1}(\s^\omega,\s)_\bb{Q},
\]
and we may identify the fibre of this projection with the fibre of the functor $P_{d-1}$, which is the $\infty$-category of $d$-homogeneous functors. It follows that the $\infty$-category of $d$-excisive functors splits as 
\[
\exc{d}(\s^\omega,\s)_\bb{Q}  \simeq \homog{d}(\s^\omega,\s)_\bb{Q} \times \exc{d-1}(\s^\omega,\s)_\bb{Q}.
\]
The splitting as a product of $\infty$-categories of homogenous functors follows by reverse induction and the observation that there is an equivalence $\homog{1}(\s^\omega,\s)_\bb{Q} \simeq \exc{1}(\s^\omega,\s)_\bb{Q}$ since all functors are assumed reduced. 
\end{proof}

The proof of~\cref{thm: splitting of exc} has the following immediate corollary, which amounts to the observation that rational homogeneous functor $D_dh_\bb{S}$ admits the structure of a commutative algebra object in $\Fun_\ast(\s^\omega,\s)_\bb{Q}$. 

\begin{cor}\label{homogeneous modules}
There is an equivalence 
\[
\homog{d}(\s^\omega,\s)_\bb{Q} \simeq \sf{Mod}_{\Fun_\ast(\s^\omega, \s)_\bb{Q}}(D_dh_\bb{S}),
\]
of $\infty$-categories.
\end{cor}

\begin{rem}
There is \emph{a priori} no reason for this result to hold integrally: $D_d$ is not a symmetric monoidal endofunctor of $\Fun_\ast(\s^\omega, \s)$ since $D_d$ does not preserve the symmetric monoidal unit of $\Fun_\ast(\s^\omega, \s)$.
The functor $D_dh_\bb{S}$ is an idempotent coalgebra object in $\Fun_\ast(\s^\omega,\s)$ since the colocalization $D_d: \exc{d}(\s^\omega,\s) \to \homog{d}(\s^\omega,\s)$ is smashing, i.e., given by Day convolution with $D_dh_\bb{S}$. In the rational setting, the identification of $D_dh_\bb{S}$ with the localization of $P_dh_\bb{S}$ at the idempotent $e_d$ induces an idempotent algebra structure on $D_dh_\bb{S}$ which is compatible with the idempotent coalgebra structure. This algebra structure cannot lift to an integral algebra structure as if it did, the splitting of~\cref{thm: splitting of exc} would also hold integrally.
\end{rem}

\section{An algebraic model for rational excisive functors}
One of the major theorems of Goodwillie calculus is the classification of homogeneous functors. To every functor, $F: \s \to \s$ Goodwillie assigns a symmetric sequence $\partial_\ast F =\{\partial_d F\}_{d \in \bb{N}}$ of spectra, in which we call the $d$-th term $\partial_d F$, the \emph{$d$-th derivative of $F$.} When the input functor is $d$-homogeneous, Goodwillie shows that it is completely determined by its $d$-th derivative. By~\cite[Theorem 3.5 and \S5]{GoodCalcIII}, there is an equivalence
\[
\homog{d}(\s^\omega,\s) \simeq \s^{B\Sigma_d}
\]
between the $\infty$-category of $d$-homogeneous endofunctors of spectra and the $\infty$-category of spectra with an action of $\Sigma_d$, the symmetric group on $d$ letters. Applying this classification of homogeneous functors to the splitting of~\cref{thm: splitting of exc} we obtain the following.

\begin{cor}
There is an equivalence
\[
\exc{d}(\s^\omega,\s)_\bb{Q} \simeq \prod_{1 \leq i \leq d} \s_\bb{Q}^{B\Sigma_i}
\]
of $\infty$-categories.
\end{cor}

Shipley's~\cite{Shipley} equivalence between $HR$-module spectra and chain complexes of $R$-modules together with the identification of $H\bb{Q}$-module spectra with rational spectra induce the algebraic model of~\cref{cor: main cor} upon the observation that there is an induced equivalence
\[
\s_\bb{Q}^{B\Sigma_i} = \Fun(B\Sigma_i, \s_\bb{Q}) \simeq \Fun(B\Sigma_i, \sf{Ch}_\bb{Q}) \simeq \sf{Ch}(\bb{Q}[\Sigma_i]). 
\]

\begin{cor}
There is an equivalence
\[
\exc{d}(\s^\omega,\s)_\bb{Q} \simeq \prod_{1 \leq i \leq d} \sf{Ch}(\bb{Q}[\Sigma_i])
\]
of $\infty$-categories. 
\end{cor}

We call the right-hand side an \emph{algebraic model} of the $\infty$-category of rational $d$-excisive functors from spectra to spectra.

\bibliography{references}
\bibliographystyle{alpha}
\end{document}